\documentclass[a4paper,11pt]{article}
\usepackage{amsmath, amsthm, amssymb}

\def\R{{\mathbb R}}

\def\C{{\mathbb C}}
\def\P{{\mathbb P}}
\def\F{{\mathcal F}}
\def\E{{\mathbb E}}

\newtheorem{thm}{Theorem}[section]
\newtheorem{cor}[thm]{Corollary}
\newtheorem{lem}[thm]{Lemma}

\theoremstyle{remark}
\newtheorem{rem}[thm]{Remark}

\theoremstyle{definition}
\newtheorem{dfn}[thm]{Definition}

\numberwithin{equation}{section}

\title{A variant of the Johnson-Lindenstrauss lemma for circulant matrices}

\author{Jan Vyb\'\i ral \\
Radon Institute for Computational and Applied Mathematics (RICAM)\\
Austrian Academy of Sciences\\
Altenbergerstra\ss e 69\\
A-4040 Linz, Austria\\
email:\ jan.vybiral@oeaw.ac.at}

\begin{document}
\maketitle

\begin{abstract}
We continue our study of the Johnson-Lindenstrauss lemma and its connection
to circulant matrices started in \cite{HV}.
We reduce the bound on $k$ from $k=O(\varepsilon^{-2}\log^3n)$ proven there
to $k=O(\varepsilon^{-2}\log^2n)$.
Our technique differs essentially from the one used in \cite{HV}. 
We employ the discrete Fourier transform and singular value decomposition
to deal with the dependency caused by the circulant structure.
\end{abstract}

{\bf AMS Classification: }{52C99, 68Q01}

{\bf Keywords and phrases: }{Johnson-Lindenstrauss lemma, circulant matrix, 
discrete Fourier transform, singular value decomposition}

\section{Introduction}

Let $x^1,\dots,x^n\in \R^d$ be $n$ points in the $d$-dimensional Euclidean space $\R^d$.
The classical Johnson-Lindenstrauss
lemma tells that, for a given $\varepsilon\in(0,\frac 12)$ and a natural number
$k=O(\varepsilon^{-2}\log n)$, there exists a linear map $f:\R^d\to\R^k$, such that
$$
 (1-\varepsilon)||x^j||_2^2\le ||f(x^j)||_2^2\le (1+\varepsilon)||x^j||_2^2
$$
for all $j\in\{1,\dots,n\}.$ 

Here $||\cdot||_2$ stands for the Euclidean norm in $\R^d$ or $\R^k$, respectively.
Furthermore, here and any time later, the condition $k=O(\varepsilon^{-2}\log n)$ means, that there is an absolute
constant $C>0$, such that the statement holds for all natural numbers $k$ with
$k\ge C\varepsilon^{-2}\log n$. We shall also always assume, that $k\le d$. Otherwise, 
the statement becomes trivial.

The original proof of this fact was given by Johnson and Lindenstrauss in \cite{JL}.
We refer to \cite{DG} for a beautiful and self-contained proof.
Since then, it has found many applications for example in algorithm design. 
These applications inspired
numerous variants  and improvements of the Johnson-Lindenstrauss lemma, 
which try to minimize the computational costs of $f(x)$,
the memory used, the number of random bits used and to
simplify the algorithm to allow an easy implementation.
We refer to \cite{IM,A,AC2,AC,M} for details and to \cite{M} for a nice description
of the history and the actual ``state of the art''.

All the known proofs of the Johnson-Lindenstrauss lemma work with random matrices
and proceed more or less in the following way. One considers a probability measure
$\P$ on a some subset ${\mathcal P}$ of all $k\times d$ matrices (i.e. all linear
mappings $\R^d\to\R^k$). The proof of the Johnson-Lindenstrauss lemma 
then emerges by some variant of the following two estimates
$$
\P\biggl(f\in{\mathcal P}:||f(x)||_2^2\ge 1+\varepsilon\biggr)<1-\frac {1}{2n}
$$
and
$$
\P\biggl(f\in{\mathcal P}:||f(x)||_2^2\le 1-\varepsilon\biggr)<1-\frac {1}{2n},
$$
which have to be proven for all unit vectors $x\in\R^d$,
and a simple union bound
over all points $x^j/||x^j||_2, j=1,\dots,n$. Here and later on
we assume, without loss of generality, that
$x^j\not =0$ for all $j=1,\dots,n.$

The best known construction of $f$ (according to the properties mentioned above)
was given by Ailon and Chazelle in \cite{AC2} with an improvement due to Matou\v{s}ek, cf. \cite{M}. 
It states, that $f$ may be given as a composition of a sparse matrix, certain
random Fourier matrix and a random diagonal matrix.
Although it provides a good computational time of $f(x)$ 
(with high probability $f(x)$ may be computed using 
$O(d\log d+\min\{d\varepsilon^{-2}\log n,\varepsilon^{-2}\log^3n\})$ operations),
it still needs, that each coordinate of the $k\times d$ matrix is generated
independently. In \cite{HV}, we studied a different construction of $f$,
namely the possibility of a composition of a random circulant matrix with
a random diagonal matrix. As a multiple of a circulant matrix may be implemented
with the help of a discrete Fourier transform, it provides the running time
of $O(d\log d)$, requires less randomness (only $2d$ compared to $kd$ or $(k+1)d$
used earlier) and allows a very simple implementation, as the Fast Fourier Transform
is a part of every standard mathematical software package.

The main difference between this approach and all the other constructions available in 
the literature so far is that the components of $f(x)$ are now no longer independent
random variables. Decoupling this dependence, we were able to prove in \cite{HV} the
Johnson-Lindenstrauss lemma for composition of a random circulant matrix
and a random diagonal matrix, but only for $k=O(\varepsilon^{-2}\log^3n)$.
It is the main aim of this note to improve this bound to $k=O(\varepsilon^{-2}\log^2n)$.
This comes essentially closer to the standard bound $k=O(\varepsilon^{-2}\log n)$.
Reaching this optimal bound (and keeping the control of the constants involved) 
remains an open problem and a subject of
a challenging research.

We use a completely different technique here. We use the discrete Fourier transform
and the singular value decomposition of circulant matrices. That is the reason,
why we found it more instructive to
state and prove our variant of Johnson-Lindenstrauss lemma for complex vectors
and Gaussian random variables.
As a corollary, we obtain of course a corresponding real version.
%With some modifications, one could modify the proof and deal directly with
%real Euclidean spaces.

To state our main result, we first fix some notation. Let
\begin{itemize}
\item $\varepsilon\in (0,\frac 12)$,
\item $n\ge d$ be natural numbers,
\item $x^1,\dots,x^n\in\C^d$ be $n$ arbitrary points in $\C^d$,
\item $k=O(\varepsilon^{-2}\log^2 n)$ be a natural number smaller then $d$,
\item $a=(a_0,\dots,a_{d-1})$ be independent complex Gaussian variables, cf. Definition \ref{dfn1},
\item $\varkappa=(\varkappa_0,\dots,\varkappa_{d-1})$ be independent Bernoulli variables.
\end{itemize}
We denote by $M_{a,k}$ and $D_\varkappa$ the partial random circulant matrix
and the random diagonal matrix, respectively, cf. Definition \ref{dfn2} for details.
\begin{thm}\label{thm1}
The mapping $f:\C^d\to \C^k$ given by $f(x)=\frac{1}{\sqrt{2k}}M_{a,k}D_\varkappa x$ satisfies
$$
 (1-\varepsilon)||x^j||_2^2\le ||f(x^j)||_2^2\le (1+\varepsilon)||x^j||_2^2
$$
for all $j\in\{1,\dots,n\}$ with probability at least 2/3.
Here $||\cdot||_2$ stands for the $\ell_2$-norm in $\C^d$ or $\C^k$, respectively.
\end{thm}

For reader's convenience, we formulate also a variant of Theorem \ref{thm1}, which deals with
real Euclidean spaces.

\begin{cor}\label{cor1}
Let $\varepsilon\in (0,\frac 12)$, $n\ge d$ be natural numbers,
and let $x^1,\dots,x^{n}\in\R^{2d}$ be $n$ arbitrary points in $\R^{2d}$.
Let $\alpha_0,\dots,\alpha_{d-1},\beta_0,\dots,\beta_{d-1}$ be $2d$ independent real Gaussian variables
and let $\varkappa=(\varkappa_0,\dots,\varkappa_{d-1})$ be independent Bernoulli variables.

If $k=O(\varepsilon^{-2}\log^2 n)$ is a natural number,
then the mapping $f:\R^{2d}\to \R^{2k}$ given by 
$$
f(x)=\frac{1}{\sqrt{2k}}
\left(\begin{matrix}M_{\alpha,k}&-M_{\beta,k}\\
M_{\beta,k}&M_{\alpha,k}\end{matrix}\right)
\left(\begin{matrix}D_{\varkappa}&0\\
0&D_{\varkappa}\end{matrix}\right)x
$$ 
satisfies
$$
 (1-\varepsilon)||x^j||_2^2\le ||f(x^j)||_2^2\le (1+\varepsilon)||x^j||_2^2
$$
for all $j\in\{1,\dots,n\}$ with probability at least 2/3.
Here $||\cdot||_2$ stands for the $\ell_2$-norm in $\R^{2d}$ or $\R^{2k}$, respectively.
\end{cor}
The proof follows trivially from Theorem \ref{thm1} by considering complex Gaussian variables
$a=(\alpha_0+i\beta_0,\dots,\alpha_{d-1}+i\beta_{d-1})$ and complex vectors 
$y^j=(x^j_0+ix^j_{d},\dots,x^j_{d-1}+ix^j_{2d-1})\in \C^d$, $j=1,\dots,n$.

\section{Used techniques}
Let us give a brief overview of techniques used in the proof of Theorem \ref{thm1}.
We shall list only those few properties needed in the sequel.

\subsection{Discrete Fourier transform}
Our main tool in this note is the discrete Fourier transform. If $d$ is a natural number, then
the discrete Fourier transform $\F_d:\C^d\to\C^d$ is defined by
$$
(\F_d x)(\xi)=\frac{1}{\sqrt{d}}\sum_{u=0}^{d-1}x_u \exp\Bigl(-\frac{2\pi iu\xi}{d}\Bigr).
$$
With this normalisation, $\F_d$ is an isomorphism of $\C^d$ onto itself. The inverse discrete
Fourier transform is given by
$$
(\F^{-1}_d x)(\xi)=\frac{1}{\sqrt{d}}\sum_{u=0}^{d-1}x_u \exp\Bigl(\frac{2\pi i u\xi}{d}\Bigr).
$$
Observe, that the matrix representation of $\F^{-1}_d$ is the conjugate transpose
of the matrix representation of $\F_d$, i.e. $\F^{-1}_d=\F_d^*$.

\subsection{Circulant matrices}
\begin{dfn}\label{dfn1}
Let $\alpha$ and $\beta$ be independent real Gaussian random variables with
$$
\E\alpha=\E\beta=0 \quad\text{and}\quad \E|\alpha|^2=\E|\beta|^2=1.
$$ 
Then we call
$$
a=\alpha+i\beta
$$
\emph{complex Gaussian variable.}
\end{dfn}
Let us note, that if $a$ is a complex Gaussian variable, then
$$
\E a=\E \alpha + i \E \beta = 0 \quad\text{and}\quad 
\E |a|^2 = \E \alpha^2 + \E \beta^2 = 2.
$$
\begin{dfn}\label{dfn2}
(i) Let $k\le d$ be natural numbers. Let $a=(a_0,\dots,a_{d-1})\in \C^d$ be a fixed complex vector.
We denote by $M_{a,k}$ the partial circulant matrix
$$M_{a,k}=\left(
\begin{matrix}
a_0 & a_1 & a_2 & \dots & a_{d-1}\\
a_{d-1} & a_0 & a_1 & \dots & a_{d-2}\\
a_{d-2} & a_{d-1} & a_{0} &\dots & a_{d-3}\\
\vdots & \vdots & \vdots &\ddots & \vdots\\
a_{d-k+1} & a_{d-k+2} & a_{d-k+3} & \dots & a_{d-k}
\end{matrix}
\right)\in \C^{k\times d}.
$$
If $k=d$, we denote by $M_{a}=M_{a,d}$ the full circulant matrix.
This notation extends naturally to the case, when $a=(a_0,\dots,a_{d-1})$ are
independent complex Gaussian variables.

(ii) If $\varkappa=(\varkappa_0,\dots,\varkappa_{d-1})$ are independent Bernoulli
variables, we put
$$
D_{\varkappa}={\rm diag}(\varkappa):=\left(\begin{matrix}
\varkappa_0 & 0  & \dots & 0\\
0 & \varkappa_1  & \dots & 0\\
\vdots & \vdots  &\ddots & \vdots\\
0 & 0 & \dots & \varkappa_{d-1}
\end{matrix}\right)\in\R^{d\times d}.
$$
\end{dfn}
Of course, $D_\varkappa:\C^d\to\C^d$ is also an isomorphism.

The fundamental connection between discrete Fourier transform and circulant matrices
is given by
\begin{equation}\label{eq:1}
M_a=\F_d\, {\rm diag}(\sqrt d\F_d a) \F^{-1}_d,
\end{equation}
which may be verified by direct calculation.
Hence every circulant matrix may be diagonalised with the use of a discrete Fourier
transform, its inverse and a multiple of the discrete Fourier transform of its first row.

\subsection{Singular value decomposition}
The last tool needed in the proof is the singular value decomposition. Let $M:\C^d\to \C^k$
be a $k\times d$ complex matrix with $k\le d$. Then there exists a decomposition
$$
M=U\Sigma V^*,
$$
where $U$ is a $k\times k$ unitary complex matrix,
$\Sigma$ is a $k\times k$ diagonal matrix with nonnegative entries on the diagonal,
$V$ is a $d\times k$ complex matrix with $k$ orthonormal columns
and $V^*$ denotes the conjugate transpose of $V$. Hence $V^*$ has $k$ orthonormal rows.
The entries of $\Sigma$ are the singular values of $M$, namely the square roots of the eigenvalues of $MM^*$.

If $a=(a_0,\dots,a_{d-1})\in \C^d$ is a complex vector and $M_a$ is the corresponding circulant matrix,
then its singular values may be calculated using \eqref{eq:1}. We obtain
\begin{align*}
M_aM_a^*&=\F_d {\rm diag}(\sqrt d\F_d a) \F^{-1}_d [\F_d {\rm diag}(\sqrt d\F_d a) \F_d^{-1}]^*
=\F_d {\rm diag}(\sqrt d\F_d a) {\rm diag}(\overline{\sqrt d\F_d a}) \F_d^{-1}\\
&=\F_d {\rm diag}(d|\F_d a|^2) \F_d^{-1}.
\end{align*}
Hence, the singular values of $M_a$ are $\{\sqrt d|(\F_d a)(\xi)|\}_{\xi=0}^{d-1}.$

The action of an arbitrary projection onto a vector of independent real Gaussian variables 
is very well known. It may be described as follows.

\begin{lem}\label{lem3}
Let $a=(a_0,\dots,a_{d-1})$ be independent real Gaussian variables.
Let $k\le d$ be a natural number and let $x^1,\dots,x^k$ be mutually orthogonal unit vectors
in $\R^d$. Then
$$
\{\langle a,x^j\rangle\}_{j=1}^k
$$
is equidistributed with a $k$-dimensional vector of independent real Gaussian variables.
\end{lem}
A direct calculation shows, that Lemma \ref{lem3} holds also for complex vectors $a$ and
$x^1,\dots,x^k$. We present the following formulation of this fact.
\begin{lem}\label{lem4}
Let $a=(a_0,\dots,a_{d-1})$ be independent complex Gaussian variables. Let $W$ be a $k\times d$
matrix with $k$ orthonormal rows. Then $Wa$ is equidistributed
with a $k$-dimensional vector of independent complex Gaussian variables.
\end{lem}

\section{Proof of Theorem \ref{thm1}}

We shall need the following statement, which describes the preconditioning role of the diagonal matrix $D_\varkappa$.
A similar fact has been used also in \cite{AC2}. Nevertheless, using discrete Fourier transform
instead of a Hadamard matrix does not pose any restrictions on the underlying dimension $d$.
Without repeating the details, we point out, that we discussed briefly in \cite[Remark 2.5]{HV},
why this preconditioning may not be omitted.
\begin{lem}\label{lem1}
Let $n\ge d$ be natural numbers and let $x^1,\dots,x^n\in\C^d$ be complex vectors. 
Let $\varkappa=(\varkappa_0,\dots,\varkappa_{d-1})$ be independent Bernoulli variables.
Then there is an absolute constant $C>0$, such that with probability at least $5/6$
\begin{equation}\label{eq:3.1}
||\F_d\, D_{\varkappa}(x^j)||_\infty \le \frac{C\,\sqrt {\log n}}{\sqrt d} \cdot ||x^j||_2
\end{equation}
holds for all $j\in\{1,\dots,n\}.$
\end{lem}
\begin{proof}
Let $x=\alpha+i\beta$ be a unit complex vector in $\C^d$. 
We put $y=(y_0,\dots,y_{d-1})=\F_d \, D_{\varkappa}(x)$. Then we may estimate
\begin{equation}\label{eq:3.11}
\P_\varkappa(|y_l|>s)\le 2\P_\varkappa(\Re y_l>\frac{s}{\sqrt 2})+2\P_\varkappa(\Im y_l>\frac{s}{\sqrt 2}),
\quad l=0,\dots, d-1,
\end{equation}
where
$$
\Re y_l=\frac {1}{\sqrt d}\sum_{u=0}^{d-1}\varkappa_u[\alpha_u\cos(2\pi lu/d)+\beta_u\sin(2\pi lu/d)]
$$
and
$$
\Im y_l=\frac {1}{\sqrt d}\sum_{u=0}^{d-1}\varkappa_u[\beta_u\cos(2\pi lu/d)-\alpha_u\sin(2\pi lu/d)]
$$
are the real and the imaginary part of $y_l$, respectively.

Using the Markov's inequality and a real parameter $t>0$, which is at our disposal, we may
proceed in a standard way:
\begin{align*}
\P_\varkappa\Bigl(\Re y_l>\frac{s}{\sqrt 2}\Bigr)&=
\P_\varkappa\Bigl(\exp(t\Re y_l-\frac{st}{\sqrt 2})>1\Bigr)\\
&\le \exp\Bigl(-\frac{st}{\sqrt 2}\Bigr)\E_\varkappa\exp(t\Re y_l)\\
&= \exp\Bigl(-\frac{st}{\sqrt 2}\Bigr) \prod_{u=0}^{d-1}\cosh\Bigl[\frac{t}{\sqrt d}[\alpha_u\cos(2\pi lu/d)+\beta_u\sin(2\pi lu/d)]\Bigr]\\
&\le \exp\Bigl(-\frac{st}{\sqrt 2}\Bigr) \prod_{u=0}^{d-1} \exp\Bigl(\frac{t^2}{2d}[\alpha_u\cos(2\pi lu/d)+\beta_u\sin(2\pi lu/d)]^2\Bigr)\\
&\le \exp\Bigl(-\frac{st}{\sqrt 2}\Bigr) \prod_{u=0}^{d-1} \exp\Bigl(\frac{t^2}{2d}[\alpha_u^2+\beta_u^2]\Bigr)=
\exp\Bigl(-\frac{st}{\sqrt 2}+\frac{t^2}{2d}\Bigr).
\end{align*}
We have used the inequality $\cosh(v)\le \exp(v^2/2)$, which holds for all $v\in\R$,
and the inequality between geometric and quadratic means.
For the optimal $t=\frac{sd}{\sqrt 2}$, this is equal to $\exp(-\frac{s^2d}{4})$.

As the second summand in \eqref{eq:3.11} may be estimated in the same way, we obtain
\begin{equation}\label{eq:3.12}
\P_\varkappa(|y_l|>s)\le 4 \exp\Bigl(-\frac{s^2d}{4}\Bigr),\quad l=0,\dots, d-1.
\end{equation}
Choosing $s=O(d^{-1/2}\sqrt{\log n})$ and applying the union bound over all 
$nd\le n^2$ components of $\{\F_d\, D_{\varkappa}(x^j/||x^j||_2)\}_{j=1}^n$,
we obtain the result.
\end{proof}

{\it Proof of Theorem \ref{thm1}}

Let us choose a vector $\varkappa=(\varkappa_0,\dots,\varkappa_{d-1})\in\{-1,+1\}^d$, 
such that \eqref{eq:3.1} holds.
According to the Lemma \ref{lem1} this happens with probability at least $5/6$.

Let us take $\tilde x=\frac{x^j}{||x^j||_2}$ for any fixed $j=1,\dots,n.$ We show, that
there is an absolute constant $c>0$, such that
\begin{equation}\label{eq:3.2}
\P_a\bigl(||M_{a,k}D_\varkappa \tilde x||_2^2\ge 2(1+\varepsilon) k\bigr)\le 
\exp\Bigl(-\frac{ck\varepsilon^{2}}{\log n}\Bigr)
\end{equation}
and
\begin{equation}\label{eq:3.3}
\P_a\bigl(||M_{a,k}D_\varkappa \tilde x||_2^2\le 2(1-\varepsilon) k\bigr)\le 
\exp\Bigl(-\frac{ck\varepsilon^{2}}{\log n}\Bigr)
\end{equation}
holds. From \eqref{eq:3.2} and \eqref{eq:3.3}, Theorem \ref{thm1} follows again by a union 
bound over all $j= 1,\dots, n.$

Let $y^j=S^j(D_{\varkappa} \tilde x)\in\C^d$, $j=0,\dots,k-1$, where $S$ is the shift operator defined by
$$
S:\C^d\to \C^d, \quad S(z_0,\dots,z_{d-1})=(z_1,\dots,z_{d-1},z_0).
$$
We denote by $Y$ the $k\times d$ matrix with rows $y^0,\dots,y^{k-1}$.

Then it holds
\begin{equation*}
||M_{a,k}D_\varkappa \tilde x||_2^2=
\sum_{j=0}^{k-1}\bigl|\sum_{u=0}^{d-1}a_{(u-j)\,{\rm mod}\, d\,}\varkappa_u\tilde x_u\bigr|^2
=\sum_{j=0}^{k-1}\bigl| \sum_{u=0}^{d-1} y^j_u a_u\bigr|^2
=||Ya||_2^2.
\end{equation*}

Let $Y=U\Sigma V^*$ be the singular value decomposition of $Y$. As mentioned above, $b:=V^*a$
is a $k$-dimensional vector of independent complex Gaussian variables. Hence,
$$
||Y a||_2^2=||U\Sigma V^* a||_2^2=||U\Sigma b||_2^2=||\Sigma b||_2^2=\sum_{j=0}^{k-1}\lambda^2_j |b_j|^2,
$$
where $\lambda_j, j=0,\dots, k-1$, are the singular values of $Y$. Let us denote $\mu_j=\lambda_j^2$.
Then 
$$
||\mu||_1=\sum_{j=0}^{k-1}\lambda_j^2=||Y||^2_{F}=k,
$$
where $||Y||_F$ is the Frobenius norm of $Y$.

Moreover, 
\begin{align}\label{eq:3.4}
||\mu||_\infty&=||\lambda||^2_\infty
=\sup_{z\in\C^d, ||z||_2\le 1}||Yz||_2^2\\
&\notag\le \sup_{z\in\C^d, ||z||_2\le 1}||M_{D_\varkappa \tilde x}z||_2^2
= d||\F_d D_\varkappa(\tilde x)||^2_\infty \le C^2\log n,
\end{align}
where $M_{D_\varkappa \tilde x}$ stands for the $d\times d$ complex 
circulant matrix with the first row equal to $D_\varkappa \tilde x$.

This leads finally also to 
\begin{equation}\label{eq:3.5}
||\mu||_2\le \sqrt{||\mu||_1\cdot ||\mu||_\infty}\le C\sqrt{k\log n}.
\end{equation}
Then
$$
\P_a\bigl(||Ya||_2^2>2(1+\varepsilon)k\bigr)=
\P_b\biggl(\sum_{j=0}^{k-1}\mu_j(|b_j|^2-2)>2\varepsilon k\biggr).
$$
We denote
$$
Z:=\sum_{j=0}^{k-1}\mu_j(|b_j|^2-2).
$$
The complex version of Lemma 1 from Section 4.1 of \cite{LM} (cf. also Lemma 2.2 of \cite{M})
states that
\begin{equation}\label{eq:3.8}
\P_b(Z\ge 2\sqrt 2||\mu||_2\sqrt t+2||\mu||_\infty t)\le\exp(-t).
\end{equation}
Using \eqref{eq:3.4} and \eqref{eq:3.5}, we arrive at
$$
\P_b(Z\ge 2\sqrt 2 C \sqrt{tk\log n}+2C^2 t\log n)\le\exp(-t).
$$
Choosing $t= \frac{c'k\varepsilon^2}{C^2 \log n}$ for $c'>0$ small enough, we get
$$
\P_b(Z\ge 2\varepsilon k)\le \exp\Bigl(-\frac{ck\varepsilon^2}{\log n}\Bigr).
$$
This finishes the proof of \eqref{eq:3.2}. Let us note, that \eqref{eq:3.3} follows in the same manner with
\eqref{eq:3.8} replaced by
\begin{equation*}
\P_b(Z\le -2\sqrt 2||\mu||_2\sqrt t)\le\exp(-t),
\end{equation*}
which may be again found in Lemma 1, Section 4.1 of \cite{LM}.

\begin{rem}
The statement and the proof of Theorem \ref{thm1} do not change, if we replace
the partial circulant matrix $M_{a,k}$ with any $k\times d$ submatrix of $M_a$.
\end{rem}

{\bf Acknowledgement:}
The author would like to thank Aicke Hinrichs for valuable comments.
The author also acknowledges the financial support provided by the 
FWF project Y 432-N15 START-Preis “Sparse Approximation and 
Optimization in High Dimensions”.


\thebibliography{99}
\bibitem{A} D.~Achlioptas, Database-friendly random projections: Johnson-Lindenstrauss with binary coins.
\emph{J. Comput. Syst. Sci.}, 66(4):671-687, 2003.

\bibitem{AC2} N.~Ailon and B.~Chazelle, Approximate nearest neighbors and the fast Johnson-Lindenstrauss transform.
In \emph{Proc. 38th Annual ACM Symposium on Theory of Computing}, 2006.

\bibitem{AC} N.~Ailon and B.~Chazelle, The fast Johnson-Lindenstrauss transform
and approximate nearest neighbors. \emph{SIAM J. Comput.} 39 (1), 302-322, 2009.

\bibitem{DG} S.~Dasgupta and A.~Gupta, An elementary proof of a theorem of Johnson and
Lindenstrauss. \emph{Random. Struct. Algorithms}, 22:60-65, 2003.

\bibitem{HV} A. Hinrichs and J. Vyb\'iral, Johnson-Lindenstrauss lemma for circulant matrices,
submitted, available on {\tt http://arxiv.org/abs/1001.4919}.

\bibitem{IM} P.~Indyk and R.~Motwani, Approximate nearest neighbors: Towards removing the curse of
dimensionality. In \emph{Proc. 30th Annual ACM Symposium on Theory of Computing}, pp. 604-613, 1998.

\bibitem{JL} W.~B.~Johnson and J.~Lindenstrauss, Extensions of Lipschitz mappings into a Hilbert space.
\emph{Contem. Math.}, 26:189-206, 1984.

\bibitem{LM} B.~Laurent and P. Massart, Adaptive estimation of a quadratic functional 
by model selection.  \emph{Ann. Statist.}  28(5):1302--1338, 2000.

\bibitem{M} J.~Matou\v{s}ek, On variants of the Johnson-Lindenstrauss lemma,
\emph{Random Struct. Algorithms}  33(2):142--156, 2008.
\end{document}